\documentclass[11pt,a4]{article}
\usepackage[colorlinks=true, pdfstartview=FitV, linkcolor=blue, citecolor=blue,urlcolor=blue]{hyperref}
\usepackage[applemac]{inputenc}
\usepackage{amssymb,amsmath,amsthm,amsfonts,amscd,amstext}
\usepackage{mathtools}
\usepackage[mathscr]{eucal}
\usepackage{multirow}
\usepackage{xypic}
\usepackage{array}
\usepackage{graphicx}
\newtheorem{thm}{Theorem}[section]
\newtheorem{lem}[thm]{Lemma}
\newtheorem{prop}[thm]{Proposition}

\newtheorem{cor}[thm]{Corollary}
\newtheorem{df}[thm]{Definition\rm}
\newtheorem{rem}{\it Remark\/}
\def\C{{\mathbb C}}  

\def\H {\mathbb H} 

\def\id{\operatorname{id}}  

\def\ker{\operatorname{ker}}  

\def\M {\mathbb M}   

\def\pr{\operatorname{pr}} 

\def\R{{\mathbb R}}    

\def\N{{\mathbb N}}    

\def\supp{\operatorname{supp}}  

\def\og{\leavevmode\raise.3ex\hbox{$\scriptscriptstyle\langle\!\langle$~}}

\def\fg{\leavevmode\raise.3ex\hbox{~$\!\scriptscriptstyle\,\rangle\!\rangle$}}
\begin{document}
\title{\bf Generalised Operations in Free Harmonic Analysis}
\author
{Roland Friedrich\thanks{Saarland University, friedrich@math.uni-sb.de. Supported by the  ERC advanced grant ``Noncommutative distributions in free probability". 
}}

\maketitle
\date{}
\begin{abstract}
This article, which is substantially motivated by the previous joint work with J.~McKay~\cite{FMcK2012}, establishes the analytic analogues of the relations we found free probability has with Witt vectors. Therefore, we first present a novel analytic derivation of an exponential map which relates the free additive convolution on $\R$ with the free multiplicative convolution on either the unit circle or $\R_+^*$, for compactly supported, freely infinitely divisible probability measures.
We then introduce several new operations on these measures, which gives rise to more extended classes of operations. Then we consider the relation with classical infinite divisibility and the geometry of the spaces involved. Finally, we discuss the general structure, using the language of operads and algebraic theories, of the various operations we defined give rise to. 
\end{abstract}

\tableofcontents
\section{Introduction}
The theory of probability measures on metric groups is a well-studied subject, cf. e.g.~\cite{P2005}. A fundamental theorem therein states that the set of probability measures $M(X)$ on a group $X$, which is also a separable metric space, becomes a topological semi-group for the operation $*$, which is defined as the convolution of measures, i.e. for $\mu,\nu\in M(X)$, one has
$$
\mu*\nu(A):=\int_X\mu(Ax^{-1})d\nu(x),\quad A\in\mathcal{B}(X),
$$
with $\mathcal{B}(X)$ the Borel $\sigma$-algebra on $X$. An important observation is that the restriction imposed by positivity, reduces the original structure on the underlying space to a substructure on the associated space of measure, i.e.,
$$
\text{group~$\rightarrow$~semi-group}.
$$
An analogous theory, where the notion of independence is replaced by freeness, was introduced and developed by D.V. Voiculescu~\cite{V1986,V1987,VDN}, D.V. Voiculescu and H. Bercovici ~\cite{BV,BV1993}, and H. Bercovici and V. Pata~\cite{BPB}, who related free harmonic analysis with classically infinitely divisible probability measures, as an isomorphism of semi-groups. 

In~\cite{BV1993} H. Bercovici and D.V. Voiculescu discussed more general convolution operations on the set of probability measures supported on $\R$. Namely, they considered selfadjoint polynomial expressions $Q(T_1,\dots,T_n)$ in $n$ non-commuting variables $T_i$, and showed that the resulting distribution $\mu$,  
$$
\mu:=Q^{\square}(\mu_1,\dots,\mu_n)
$$
with $\mu_i=\mu_{T_i}$, $i=1,\dots,n$, for free random variables $T_1,\dots,T_n$ in a $W^*$-probability space $(A,\tau)$, can be expressed by the $\mu_i$ only.

The generalisation of this line of reasoning is C.~Malle's {\em Traffic}~\cite{M2011}, who considers the operad  generated by graph polynomials with an appropriate notion of independence, called {\em Traffic freeness}.

Here we shall consider more general operations directly on proper subsets of the set of probability measures supported on $\R$, i.e. we have
$$
Q^{\square}(\mu_1,\dots,\mu_n)\rightarrow\mathcal{P}_n(\mu_1,\dots,\mu_n)
$$
where the $\mathcal{P}_n$ either define a $\mathcal{P}$-operad or  the operations of an algebraic theory.

Our motivation to consider such general operations directly on probability distributions comes from the relation J.~McKay and the author~\cite{FMcK2011,FMcK2012}  found free probability, respectively free harmonic analysis, has with complex cobordism and the ring of Witt vectors, cf.~e.g.~\cite{H}. 

A fundamental role within this framework is played by free cumulants, which have been introduced earlier by R.~Speicher~\cite{S1997} and which form the base of the combinatorial approach to free probability, as developed by A. Nica and R. Speicher~\cite{NS}.

Now, this link permitted, amongst other things, to derive a formula which linearises Voiculescu's  $S$-transform~\cite{V1987} and also to define an additional convolution operation on free cumulants which, a priori, is not induced by the underlying algebra structure. Subsequently, it was shown~\cite{FMcK2013a} that for several classes of probability distributions, this new operation indeed respects the analytic constraints given by positive definiteness, and that it induces a partial ring structure on the set of free cumulants.

Here, we not only extend our first analytic results~\cite{FMcK2013a} to the complete set of freely infinitely divisible probability measures with compact support, but also discuss (parts of) the algebra of operations on infinitely divisible probability measures in general. 

Understanding the set, respectively algebra of operations, although not traditional in probability theory, has been of primordial importance in other mathematical fields and hence intensively studied. For example, in formal language theory several important algebraic operations such as, e.g. the shuffle product, have been described which preserve the set of recognisable formal languages or in the theory of Witt vectors one has the functorial endomorphisms, e.g. Frobenius and  Verschiebung.

Therefore, our results are not only important from a general perspective but should hopefully also have fruitful applications, e.g. in random matrix theory.

Organisation of the paper. In Section 1, after giving the general motivation, the needed concepts from free probability theory are recalled.  Section 2, develops the analytic counterpart of the algebraic theory presented in~\cite{FMcK2012} for freely infinitely divisible probability measures with compact support. The first subsection discusses the holomorphic linearisation of the $S$-transform, then establishes the structure of an analytic Witt semi-ring, introduces, in analogy with the theory of Witt vectors, several endomorphisms on the space of infinitely divisible measure and then continues to discuss the geometric meaning of the Lévy-Khintchine representation. In the last subsection, then the relation with classically infinitely divisible probability measures is established via the Bercovici-Pata bijection~\cite{BPB}.  The last section, Section 3, introduces a categorical treatment, and it consists of three subsections. It  starts with an abstract discussion of convex sets and the barycentric algebra, and its relation to the operations we previously defined on infinitely divisible probability measures. In particular it highlights the fact that one can not mix two different co-ordinate systems, i.e. operations based on both moments and cumulants can not be consistently defined. Then we introduce the fundamental (finite) Giry monad and its relations to infinitely divisible measures. And finally, in the last part, we summarise the previous sections within the language of operads and the more general framework of algebraic theories.

\subsection{Preliminaries from free probability}

General information about free probability can be obtained from the monographs~\cite{NS,VDN}. Here we shall consider the one-dimensional case only.
Let $k$ be a field of characteristic zero, and $\mathbf{Alg}_k$ the category of associative and unital $k$-algebras, and $\mathbf{cAlg}_k$ the associative, commutative and unital $k$-algebras. \begin{df}
Let $A\in\mathbf{Alg}_k$, with unit $1_A$, and let $\phi:A\rightarrow k$ be a $k$-linear functional with $\phi(1_A)=1_k$. The pair $(A,\phi)$, is called a {\bf non-commutative $k$-probability space}, and the elements $a\in A$ the {\bf random variables}.
\end{df}
The algebraic $k$-probability spaces form a category, which we denote by $\mathbf{Alg_kP}$. The subcategory of commutative algebraic probability spaces is denoted by $\mathbf{cAlg_kP}$. The full subcategory of $C^*$-probability spaces $\mathbf{C^*Alg_{\C}P}$, has as objects pairs $(A,\varphi)$, where $A$ is a $C^*$-algebra and $\phi$ a state.
\begin{df}
The {\bf law} or {\bf distribution} of a random variable $a\in A$, is the $k$-linear functional $\mu_a\in k[[x]]$,
given by $\mu_a(x^n):=\phi(a^n)$, $n\in\N$.
The coefficients $(m_n(a))_{n\in\N}$ of the power series $\mu_a$ are the {\bf moments} of $a$.
\end{df}

\begin{df}[Freeness]
Let $(A,\phi)\in\mathbf{Alg_kP}$ and $I$ an index set.  A family of sub-algebras $(A_i)_{i\in I}$, with $1_A\in A_i\subset A$, $i\in I$, is called {\bf freely independent} if for all $n\in\N^*$, and $a_k\in {A_{i_k}\cap\ker(\phi)}$, $1\leq k\leq n$, also
$$
a_1\cdots a_n\in{\ker(\phi)},
$$
holds, given that $i_k\neq i_{k+1}$, $1\leq k< n$.
\end{df}
One should note that the requirement for freeness is that {\em consecutive} indices must be distinct, thus $i_j=i_{j+2}$ is possible.

Let $M_c:\mathbf{Top}_{\operatorname{lcH}}\rightarrow\mathbf{Set}$ denote the functor from the category of locally compact Hausdorff topological spaces to the category of sets, which assigns to every $T\in\mathbf{Top}_{\operatorname{lcH}}$ the set of compactly supported Borel probability measures on $T$. We note that $\mathbf{Top}_{\operatorname{lcH}}$ is dual to the category of commutative $C^*$-algebras. 

Hence, $M_c(\C)$ corresponds to the set of all compactly supported Borel probability measures on $\C$, and similarly $M_c(\R), M_c(\R^*_+)$ and $M_c(S^1)$ to the subsets of compactly supported probability measures on $\R$, $\R^*_+:=(0,\infty)$ and the unit circle $S^1$, respectively. Let $M^*(-)\subset M_c(-)$ be the set of compactly supported probability measures with non-vanishing first moment and finally, let $\operatorname{Meas_c}$ denote the functor from $\mathbf{Top}_{\operatorname{lcH}}$ to compactly supported finite Borel measures.

If  $a$ is a self-adjoint element in a $C^*$-probability space, and therefore satisfies $a=a^*$,  then there exists a unique, compactly supported probability measure $\mu_a\in M_c(\R)$ such that for all $n\in\N$,
$$
\int_{\R} x^n d\mu_a(x)=\phi(a^n)
$$
holds. Let $\boxplus$ and $\boxtimes$ denote the free additive and free multiplicative convolution, respectively.

\begin{df}
Let $\mu\in M_{c}(\R)$. If for every $n\in\N^*$, there exists a $\mu_n\in M_c(\R)$,  which satisfies 
\begin{equation*}
\mu=\underbrace{\mu_n\boxplus\dots\boxplus\mu_n}_{\text{$n$ summands}}
\end{equation*}
then $\mu$ is called {\bf $\boxplus$-infinitely divisible}. Let $\mu\in M_c(\R_+)$ or $\mu\in M^*(S^1)$. If for every $n\in\N^*$ there exists a $\mu_n\in M_c(\R_+)$ or $\mu\in M^*(S^1)$ which satisfies 
\begin{equation*}
\mu=\underbrace{\mu_n\boxtimes\dots\boxtimes\mu_n}_{\text{$n$ factors}}
\end{equation*}
then $\mu$ is called {\bf $\boxtimes$-infinitely divisible}.
\end{df}
Following the notion used in~\cite{BNT,C2014}, we denote by $\operatorname{ID}_c(\R,\boxplus)$ the set of all $\boxplus$-infinitely divisible probability measures on $\R$ with compact support, and analogously $\operatorname{ID}_c(\R_+^*,\boxtimes)$ the set of all compactly supported probability measures on $\R_+^*$ which are $\boxtimes$-infinitely divisible and $\operatorname{ID}^*(S^1,\boxtimes)$ the set of all $\boxtimes$-infinitely divisible probability measures on $S^1$ with non-vanishing first moment.

For the necessary definitions and properties concerning Voiculescu's $R$- and $S$-transform one should consult~\cite{V1986,V1987,VDN} and for its combinatorial description~\cite{NS}. Let us point out that we shall, by abuse of notation, consider $R$ either as a map into $k[[z]]$ or $zk[[z]]$, which however should be clear from the context.  

\section{Free harmonic analysis and Witt vectors}
In~\cite{FMcK2012,FMcK2013}, we derived a relation between the one-dimensional free additive and the one-dimensional free multiplicative convolution, which permits to understand  $\boxtimes$  in terms of $\boxplus$ algebraically. However, in higher dimensions such a linear relation does not hold in general, as we showed in~\cite{FMcK2013a,FMcK2013}.  So, let us first recall the corresponding statement.  
\begin{thm}[\cite{FMcK2012}]
There exists a group isomorphism $\operatorname{LOG}:(\Sigma^{\times}_1, \boxtimes)\rightarrow(\Sigma,\boxplus)$, with inverse $\operatorname{EXP}$, defined by the diagram: 
\[
\begin{xy}
  \xymatrix{
      \Lambda(\C) \ar[d]_{\frac{d}{dz}\ln}^{\text{``ghost map"}} & (\Sigma^{\times}_1, \boxtimes) \ar[l]_{S} \ar[d]^{\operatorname{LOG}} \\
                              {\C[[z]]}\ar[r]^{R^{-1}} & (\Sigma,\boxplus) 
               }
\end{xy}
\]
where $R$ and $S$ denote the $R$- and $S$-transform, respectively.
\end{thm}
Somewhat earlier, M. Mastnak and A. Nica~\cite{MN} had obtained, using Hopf algebraic methods in particular characters, a similar statement for the one-dimensional case. Subsequently, G.~Cébron~\cite{C2014} established a semi-group homomorphism for freely infinitely divisible probability distributions by applying a Lévy-Khintchine type representation, and related the classical with the free case via the Bercovici-Pata bijection. M. Anshelevich and O. Arizmendi~\cite{AA2016} then discussed a map, the {\em wrapping transformation} $W$, which is a morphism of semigroups between the set of probability measures on the real line and the circle, with $\boxplus$ and $\boxtimes$ as convolution operation, respectively, and showed that $W$ transforms the additive free, Boolean, and monotone convolution into the respective multiplicative convolutions.

Now, by using the complex analytic description of compactly supported, freely infinitely divisible probability measures, we shall give first a novel analytic derivation of an exponential morphism and then introduce some new analytic operations on measures itself.

\subsection{Holomorphic linearisation}

The derivations in this subsection are based on the analytic characterisations given in~[\cite{VDN}, Section 3.7]. As usual, let $\H$ denote the upper half-plane, $\overline{\H}$ the closed upper half-plane, $-\overline{\H}$ the closed lower half-plane, $\Re(z)$ and $\Im(z)$ denote the real and imaginary part of a complex number $z$, respectively.

We consider the following sets of germs of analytic functions, cf.~[\cite{VDN} Theorems~$3.7.2$, $3.7.3$ and $3.7.4.$]:

\begin{enumerate}
\item Let $\mathcal{R}:=\mathcal{R}_{\infty,\boxplus}(\R)$ denote the set of stalks of functions, $R(z)$, which are analytic in a neighbourhood of $(\C\setminus\R)\cup\{0\}$ (direct limit), and satisfy for all $z\in\H$: 
\begin{equation}
\label{R_upper_half_plane}
R(\bar{z})=\overline{R(z)}\quad\text{and}\quad \Im(R(z))\geq0.
\end{equation}
\item Let $\mathcal{V}:=\mathcal{V}_{\infty,\boxtimes}(\R^*_+)$ denote the set of stalks of functions, $v(z)$, which are analytic in a neighbourhood of $(\C\setminus\R)\cup[-1,0]$ (direct limit), and satisfy for all $z\in\H$: 
$$
v(\bar{z})=\overline{v(z)}\quad\text{and}\quad \Im(v(z))\leq0.
$$
\item Let $\mathcal{U}:=\mathcal{U}_{\infty,\boxtimes}(S^1)$ denote the sheaf of functions, $u(z)$, which are analytic in the shifted half-plane ${H}_{-1/2}:=\{z\in\C~|\Re(z)>-\frac{1}{2}\}$ and satisfy for all $z\in H_{-1/2}$:
$$
\Re(u(z))\geq 0.
$$
\end{enumerate}

\begin{lem}
\label{convex_cones}
The following properties hold:
\begin{enumerate}
\item  $(\mathcal{R},+)$, $(\mathcal{V},+)$ and $(\mathcal{U},+)$ are {\em commutative monoids} (semi-groups), with neutral element $0$. 
\item The sets $\mathcal{R}$, $\mathcal{V}$ and $\mathcal{U}$ are {\em convex cones}, i.e.  if  $f,g\in \mathcal{R}, \mathcal{V}, \mathcal{U}$ and $\alpha,\beta\in\R_+$, then $\alpha f+\beta g\in\mathcal{R}, \mathcal{V}, \mathcal{U}$. 
\end{enumerate}
\end{lem}
\begin{proof}
We only have to show the second statement, as the first one follows from $\alpha=\beta=1$ and $\alpha=0$ or $\beta=0$, respectively. 

Let us consider the situation for $\mathcal{R}$. If $f(z),g(z)\in\mathcal{R}$ and $\alpha,\beta\geq0$ then $\alpha f(\overline{z})=\overline{\alpha f(z)}$ and $\Im(\alpha f(z))=\alpha\Im(f(z))\geq0$ for $z\in\H$. Further, $\alpha f(\overline{z})+\beta g(\overline{z})=\overline{\alpha f(z)}+\overline{\beta g(z)}$ and $\Im(\alpha f(z)+\beta g(z))=\alpha\Im(f(z))+\beta\Im(g(z))\geq0$.

The remaining cases are treated similarly. 
\end{proof}

\begin{lem}
$(\operatorname{ID}_c(\R,\boxplus),\boxplus)$ is a commutative monoid with neutral element $\delta_0$. Further, it is a sub-monoid of $(M_c(\R),\boxplus)$ and the $R$-transform is an isomorphism of commutative monoids:
$$
R:(\operatorname{ID}_c(\R,\boxplus),\boxplus)\rightarrow(\mathcal{R}_{\infty,\boxplus}(\R),+).
$$ 

\end{lem}
\begin{proof}
The first statement follows from 
$$
\mu\boxplus\nu=\mu_n^{\boxplus n}\boxplus\nu_n^{\boxplus n}=(\underbrace{\mu_n\boxplus\nu_n}_{\in M_c(\R)})^{\boxplus n}
$$
for some $\mu_n,\nu_n\in M_c(\R)$, $n\in\N^*$.
\end{proof}
\begin{lem}[Holomorphic linearisation of $\boxtimes$]
We have:
\begin{enumerate}
\item
There exists an isomorphism of commutative semi-groups, 
$$
\operatorname{EXP}:(\mathcal{V},+)\rightarrow (\operatorname{ID}_c(\R_+^*,\boxtimes),\boxtimes),
$$
given by
$v(z)\mapsto S^{-1}(e^{v(z)})$, which linearises $\boxtimes|_{\R^*_+}$.

\item
There exists an isomorphism of commutative semi-groups, 
$$
\operatorname{EXP}:(\mathcal{U},+)\rightarrow (\operatorname{ID}^*(S^1,\boxtimes),\boxtimes),
$$
given by $u(z)\mapsto S^{-1}(e^{u(z)})$, which linearises $\boxtimes|_{S^1\cap{M}^*(\C)}$.

\end{enumerate}
\end{lem}
\begin{proof}

$$
S_{\mu\boxtimes\nu}(z)=S_{\mu}\cdot S_{\nu}=e^{v(z)}\cdot e^{u(z)}=e^{v(z)+u(z)}.
$$
\end{proof}
As next we shall relate the free additive convolution on the real line with the free multiplicative convolution on the unit circle.
\begin{thm}
There exists a morphism of commutative semi-groups
$$
\operatorname{EXP}:(\operatorname{ID}_c(\R,\boxplus),\boxplus)\rightarrow (\operatorname{ID}^*(S^1,\boxtimes),\boxtimes)
$$
given by 
$$
\mu\mapsto S^{-1}(e^{-iR_{\mu}(i(z+1/2))}).
$$
\end{thm}
\begin{proof}
Let $\varphi:H_{-1/2}\rightarrow\H$ be the conformal map $z\mapsto\varphi(z):=i(z+\frac{1}{2})$. It induces an injective map $\hat{\varphi}:\mathcal{R}\rightarrow\mathcal{U}$, which associates to $R_{\mu}\in \mathcal{R}_{\infty,\boxplus}(\R)$, the holomorphic map 
$$
u(z):=-iR_{\mu}(i(z+\frac{1}{2})).
$$
Indeed, $u(z)\in\mathcal{U}$, as for $\Re(z)>-\frac{1}{2}$ we have $\Im(i(z+\frac{1}{2}))>0$, and so, by assumption, $\Im(R_{\mu}(i(z+\frac{1}{2})))\geq0$, from which we obtain that 
$$
\Re(-iR_{\mu}(i(z+\frac{1}{2})))=\Im(R_{\mu}(i(z+\frac{1}{2})))\geq 0.
$$ 
Now, from~[\cite{VDN}~Theorem~$3.7.4$] the first claim follows.

For $\delta_0$ we have $R_{\delta_0}(z)\equiv0$ which is mapped onto $1$ and hence corresponds to $\delta_{e^0}$.  
Next, for $\mu,\nu\in M_{\infty,\boxplus}(\R)$ we have $R_{\mu}+R_{\nu}=R_{\mu\boxplus\nu}$, and further
$$
\left(-iR_{\mu}(i(z+\frac{1}{2}))\right)+\left(-iR_{\nu}(i(z+\frac{1}{2}))\right)=-i(R_{\mu}+R_{\nu})\circ\varphi(z)=-i(R_{\mu\boxplus\nu})\circ\varphi(z),
$$
from which the property of a semi-group morphism follows.
\end{proof}
We have the following examples: 
\begin{itemize}
\item For $R(z)=bz$, i.e. the semi-circular law of radius $2\sqrt{b}$ and centred at the origin, we get $z\mapsto e^{bz+\frac{b}{2}}$, i.e. the {\bf free Brownian motion}.
\item For $R(z)=a$, i.e. the {\bf Dirac measure} $\delta_a$ at $a$, we get 
$a\mapsto e^{-ia}$.
\end{itemize}

The {\bf free Poisson distribution (fP) with rate $\lambda\geq0$ and jump size $\alpha\in\R$}, cf. e.g.~[\cite{NS}, Prop. 12.11, or \cite{VDN} p. 34],  is the limit in distribution for $N\rightarrow\infty$ of 
$$ 
\nu_{N,\lambda,\alpha}:=\left(\left(1-\frac{\lambda}{N}\right)\delta_0+\frac{\lambda}{N}\delta_{\alpha}\right)^{\boxplus N},
$$ 
with $\boxplus N$ in the exponent denoting the $N$-fold free additive self-convolution.

The ${R}$-transform of the limit $\nu_{\infty,\lambda,\alpha}:=\lim_{N\to\infty}\nu_{N,\lambda,\alpha}$ is
\begin{equation}
\label{inf_Poisson}
{R}_{\nu_{\infty,\lambda,\alpha}}(z)=\lambda\alpha\frac{1}{1-\alpha z}=\sum_{n=0}^{\infty}\lambda\alpha^{n+1}z^n=\lambda\alpha+\lambda\alpha^2 z+\lambda\alpha^3z^2+\cdots.
\end{equation}

Let $\mathcal{R}_{[-1,0]}(\R)\subset\mathcal{R}_{\infty,\boxplus}(\R)$ be the subset of stalks of $R$-transforms which are also analytic around $[-1,0]$ (direct limit set). Then $\mathcal{R}_{[-1,0]}(\R)$ is not empty, as it contains, e.g. the
\begin{itemize}
\item Dirac delta $\delta_a$,  $a\in\R$,
\item (Centred) semi-circular $a+bz$, $a\in\R$, $b>0$,
\item Free Poisson $\lambda\alpha/(1-\alpha z)$, $\lambda\in\R_+$ and $|\alpha|<1$.
\end{itemize}

The next result seems not to have been considered in the literature to this extent so far. Namely, we shall relate the free additive convolution on $\R$ with the free multiplicative convolution on $\R_+^*$.  However, the complex analytic machinery for compactly supported and freely infinitely divisible measures, only provides a statement which applies to a, nevertheless nontrivial, subset of the additive infinitely divisible measures.  But the expression we obtain suggests that there is an inherent difference concerning the relation between the two convolutions for these domains.

\begin{prop}
There exists a monomorphism of abelian semi-groups 
$$
\operatorname{EXP}: (\operatorname{ID}_c(\R,\boxplus)|_{\mathcal{R}_{[-1,0]}},\boxplus)\rightarrow (\operatorname{ID}_c(\R_+^*,\boxtimes),\boxtimes),
$$
given by $\mu\mapsto S^{-1}(e^{-R_{\mu}(z)})$.

\end{prop}
\begin{proof}
For $\delta_0$ we have $R_{\delta_0}(z)\equiv 0$ which maps to $e^0=1$, corresponding to $S_{\delta_{e^0}}(z)$.
Define a map $\mathcal{R}_{[-1,0]}\rightarrow\mathcal{V}$ by $R_{\mu}\mapsto-R_{\mu}$, which is injective. Then, for $v(z):=-R_{\mu}(z)$, we have
$$
v(\bar{z})=-R_{\mu}(\bar{z})\underbrace{=}_{\text{by $1.$}}-\overline{R_{\mu}(z)}=\overline{-R_{\mu}(z)}=\overline{v(z)},
$$ 
and for $\Im(z)>0$ we have $\Im(-R(z))=-\Im(R(z))\leq0$. Then from~[\cite{VDN}~Theorem~$3.7.3$] the first part of the proposition follows. The map $\operatorname{EXP}$ being a morphism of semi-groups now follows from  $\exp(-(R_{\mu}+R_{\nu}))=\exp(-R_{\mu\boxplus\nu})=\exp(-R_{\mu})\cdot\exp(-R_{\nu})$.
\end{proof}
\subsection{The Witt semi-ring $\boxplus$, $\boxdot$}

In order to establish in the category of probability measures an analogous structure to the ring of Witt vectors, which additionally respects positivity, we have to adapt first some of the characterisations of infinitely divisible probability measures in terms of cumulants, as described in~[\cite{NS} Lecture 13]. 

So let us recall, that a complex sequence $(s_1,s_2,s_3,\dots)$ is called {\bf conditionally positive definite} if the shifted sequence $(s_2,s_3,\dots)$ is positive definite, cf.~[\cite{NS}, Notation 13.10]. If $(s_0,s_1,s_2,\dots)$ is a positive definite sequence then the truncated sequence $(s_1,s_2,s_3,\dots)$ is conditionally positive definite. 
Namely, for all $n\in\N$, and all $a_0,\dots, a_n\in\C$, we have, by assumption, that
$$
\sum_{i,j=0}^n a_i\bar{a}_j s_{i+j}\geq0.
$$
Let $a_0\equiv0$. Then for $n\geq1$ and for all $a_1,\dots,a_n\in\C$, we have 
$$
0\leq\sum_{i,j=0}^n a_i\bar{a}_j s_{i+j}=\sum_{i,j=1}^n a_i\bar{a}_j s_{i+j},
$$
as we are restricting to the sub-matrix where one deletes the first row and first column. This is a particular case of the general fact that any principal sub-matrix of a positive definite matrix is again positive definite.

\begin{prop}
\label{Char_infinitely_div}
Let $({s_n})_{n\in\N^*}$, $s_n\in\R$, be a sequence of real numbers. Then the following  statements are equivalent.
\begin{enumerate}
\item The sequence $({s_n})_{n\in\N^*}$ is conditionally positive definite and exponentially bounded, i.e. there exists a $C>0$ such that $|s_n|\leq C^n$ for all $n\in\N^*$.
\item There exists a $\mu\in \operatorname{ID}_c(\R,\boxplus)$ such that $\kappa_n(\mu)=s_n$, $n\in\N^*$, i.e. the real sequence is given by the free cumulants of a freely, infinitely divisible and compactly supported probability measure.
\item There exists a unique pair $(\gamma,\rho)$ with $\gamma\in\R$ and $\rho\in \operatorname{Meas}_c(\R)$ such that
$$
\sum_{n=0}^{\infty} s_{n+1}z^n=\gamma+\int_{\R}\frac{z}{1-xz}d\rho(x).
$$
\end{enumerate}
\end{prop}

\begin{proof}
$2.\Rightarrow 1.:$ This follows from~[\cite{NS} Theorem 13.6 and Proposition 13.15].

$1.\Rightarrow 3.:$ This follows from~[\cite{NS} Proposition 13.14].

$3.\Rightarrow 1.:$ This follows from~[\cite{NS} Lemmas 13.13 and 13.14].

$1.\Rightarrow 2.:$
By~[\cite{NS} Proposition 13.14], there exists a finite, compactly supported, Borel measure $\rho$ on $\R$ such that 
\begin{equation}
\label{generating_function}
R(z):=\sum_{n=0}^{\infty} s_{n+1}z^n=s_1+\int_{\R}\frac{z}{1-xz}d\rho(x).  
\end{equation}

\item We show that $R(z)$ in~(\ref{generating_function}) satisfies the requirements of [\cite{VDN}~Theorem 3.7.3,~(2)].  
\begin{itemize}
\item
By using a version of the ``Differentiation Lemma", cf. e.g.~[\cite{Kle} Theorem 6.28], it follows that $R(z)$ is analytic in $\H$ and $-\H$ and in a neighbourhood of $0$. The statement for the first two regions follows directly from the integral representation. 

As the support of $\rho$ is bounded, there exists an $A>0$ such that $\supp(\rho)\subset[-A,A]$. Let $|z|<C_{\rho}:=\min(1,1/A)$. Then $|zx|<1$ for all $x\in\supp(\rho)$, and hence $R(z)$ is analytic in a disc of radius $C_{\rho}$ centred at the origin.
\item Positive imaginary part. We split the integral into real and imaginary parts. For $z:=a+ib$, we have  
$$
\frac{z}{1-xz}=\underbrace{\frac{a-x(a^2+b^2)}{(1-ax)^2+x^2b^2}}_{=:u(x)}+i\underbrace{\frac{b}{(1-ax)^2+x^2b^2}}_{=:v(x)}=u(x)+iv(x).
$$
Then 
$$
R(\overline{z})=s_1+\int_{\R} u(x)d\rho(x)-i\int_{\R}v(x)d\rho(x)=\overline{R(z)}
$$
\item For all $x\in\R$ and $b>0$, we have $v(x)>0$, and so for $z\in\H$,
$$
\Im(R(z))=\int_{\R}v(x)d\rho(x)\geq0.
$$
\end{itemize}
Hence, there exists a $\mu\in\operatorname{ID}(\R,\boxplus)$ with $R_{\mu}(z)=R(z)$.
\end{proof}

\begin{lem}[Shifted cumulants]
\label{shifted_cummulants}
Let $\mu\in \operatorname{ID}_c(\R,\boxplus)$ and $\kappa({\mu})=(\kappa_n(\mu))_{n\in\N^*}$ its free cumulant sequence. Then there exists a $\nu\in \operatorname{ID}_c(\R,\boxplus)$ such that 
$
\kappa_n({\nu})=\kappa_{n+2}({\mu}),
$ for all $n\in\N^*$.
\end{lem}
\begin{proof}
Define a sequence $(s_n)_{n\in\N^*}$ with $s_n:=\kappa({\mu})_{n+2}$ for $n\in\N^*$ which is conditionally positive definite and exponentially bounded. 
Then the claim follows from Proposition~\ref{Char_infinitely_div}.  
\end{proof}

For two power series, the point-wise or Hadamard multiplication, $\cdot_H$, is defined as:
$$
\left(\sum_{n=0}^{\infty} a_n z^n\right)\cdot_H \left(\sum_{n=0}^{\infty} b_n z^n\right):=\sum_{n=0}^{\infty} a_n  b_b z^n~.
$$

The $R$-transform is a monomorphism $R:M(\R)\rightarrow \R[[z]]$ which is given by $\mu\mapsto R_{\mu}(z)$. For $\mu,\nu\in\M_c(\R)$, we define a binary operation $\boxdot$, as follows: 
\begin{equation}
\label{boxdot}
\mu\boxdot\nu:=R^{-1}(R_{\mu}(z)\cdot_H R_{\nu}(z)),
\end{equation}
where $R^{-1}$ is the inverse of the $R$-transform. Now, the following statement establishes the analytic counterpart to the algebraic statement in~\cite{FMcK2012}, and as we noted in the Introduction, the constraint given by positivity will restrict the full structure, i.e. we shall have
$$
\text{ring~$\rightarrow$~semi-ring}.
$$

\begin{thm}
$(\operatorname{ID}_c(\R,\boxplus),\boxplus,\boxdot,\delta_0, \nu_{\infty,1,1})$ is a commutative semi-ring with the multiplicative unit given by the free Poisson distribution $\nu_{\infty,1,1}$.  
\end{thm}

\begin{proof}
We have $\nu_{\infty,1,1}\in\operatorname{ID}_c(\R,\boxplus)$ which follows from, e.g. the characterisation in [\cite{VDN}~Theorem 3.7.3,~(2)], and further $R_{\nu_{\infty,1,1}}=(1,1,1,\dots)$, which is clearly the unit for $\boxdot$.

We have to show that this product is well defined.
\begin{enumerate}
\item
Let $\mu,\nu\in\operatorname{ID}_c(\R,\boxplus)$ and  consider the corresponding free cumulant sequences $(k_n(\mu))_{n\in\N}$ and $(k_n(\nu))_{n\in\N}$. 
Both sequences are conditionally positive definite and by~[\cite{NS} Theorem 13.16, properties (1) and (2)], [\cite{NS} Theorem 13.16, property (3)] and [\cite{NS} Lemma 13.13] also exponentially bounded, i.e. $|\kappa_n(\bullet)|\leq C_{\bullet}^n$, $n\in\N$ for some constant $C_{\bullet}>0$, depending on the measures $\mu$ and $\nu$, respectively.

Hence, $|\kappa_n(\mu)\cdot\kappa_n(\nu)|\leq (C_{\mu}\cdot C_{\nu})^n$ for all $n\in\N$, i.e. the point-wise product of the free cumulants does not grow faster than exponentially. 

\item
In order to show positivity, we consider the corresponding Hankel matrices. Let $\tilde{\kappa}_n(\bullet):=\kappa_{n+2}(\bullet)$ for $n\in\N$. Then, by the Schur Product Theorem, 
$$
(\tilde{\kappa}_{i+j}(\mu)\cdot \tilde{\kappa}_{i+j}(\nu))_{0\leq i,j\leq n},
$$ 
is positive definite.  

Therefore the sequence $(h_n)_{n\in\N^*}$, with $h_1:=\kappa_1(\mu)\cdot\kappa_1(\nu)$ and $h_n:=\kappa_n(\mu)\cdot\kappa_n(\nu)$, $n\in\N^*$, is conditionally positive definite and exponentially bounded.

By Proposition~\ref{Char_infinitely_div} there exists 
a probability measure $\xi\in \operatorname{ID}_c(\R,\boxplus)$ such that $h_n=\kappa_n(\xi)$,$n\in\N^*$. Finally, let $\mu\boxdot\nu:=\xi$.
\item
Associativity and distributivity follow from the fact that the operations are given component-wise, which individually satisfy these properties. Apply $R$ to 
$(\mu\boxplus\nu)\boxdot\eta$, which by definition, yields, $R_{\mu\boxplus\nu}\cdot_H R_{\eta}=(R_{\mu}+R_{\nu})\cdot_H R_{\eta}$ from which the result $\mu\boxdot\eta\boxplus\nu\boxdot\eta$, after applying $R^{-1}$, follows.
\end{enumerate}
\end{proof}

We remark that the binary operation $\boxdot$ can be lifted via the exponential map to $\operatorname{ID}^*(S^1,\boxtimes)$ and $\operatorname{ID}_c(\R_+^*,\boxtimes)$, respectively. However, in the analytic case some additional considerations are necessary.

\begin{prop}
The family of Dirac distributions is a two-sided ideal in $\operatorname{ID}_c(\R,\boxplus)$. We have  
\begin{eqnarray*}
\delta_a\boxplus\delta_b&=&\delta_{a+b},\\
\delta_a\boxdot\delta_b&=&\delta_{ab}.
\end{eqnarray*}
Further for $\mu\in\operatorname{ID}(\R,\boxplus)$ we have
\begin{equation*}
\delta_a\boxdot\mu=\delta_{\kappa_1(\mu)\cdot a}
\end{equation*}
where $\kappa_1(\mu)$ is the first free cumulant of $\mu$.
\end{prop}
For $a,b\in\R$ and $r,s>0$, the {\bf semicircle law} centred at $a$ and of radius $r$, is defined as the distribution $\gamma_{a,r}:\C[z]\rightarrow\C$, given by
\begin{equation*}
\gamma_{a,r}(z^n):=\frac{2}{\pi r^2}\int_{a-r}^{a+r} t^n\sqrt{r^2-(t-a)^2}\,dt\qquad \forall n\in\N,
\end{equation*}
and its ${R}$-transform is 
\begin{equation}
\label{Rsemicircle}
{R}_{\gamma_{a,r}}(z)=a+\frac{r^2}{4}z~.
\end{equation}
\begin{prop}
Let $\gamma_{a,0}:=\delta_a$. The family of semicircular distributions is a two-sided ideal in $\operatorname{ID}_c(\R,\boxplus)$, which contains the Dirac distributions as sub-ideal. We have 
\begin{eqnarray}\nonumber
\gamma_{a,r}\boxplus\gamma_{b,s} & = & \gamma_{a+b,\sqrt{r^2+s^2}}, \\
\gamma_{a,r}\boxdot\gamma_{b,s} & = &  \gamma_{ab, rs/2}.
\end{eqnarray}  
\end{prop} 
\begin{proof}
We have ${R}_{\gamma_{a,r}\boxdot\gamma_{b,s}}(z)=abz+\frac{r^2 s^2}{4\cdot 4}z^2$ by component-wise multiplication of the ${R}$-transforms and the relation~(\ref{Rsemicircle}).

For any $\mu\in M_c(\R)$, $\kappa(\mu)=(\kappa_1(\mu),\kappa_2(\mu),\dots)$ with $\kappa_1(\mu)\in\R$ and $\kappa_2(\mu)\geq0$ as, e.g. by~[\cite{NS} Proposition~13.14], $\kappa_2(\mu)=m_0(\rho)=\rho(\R)\geq0$ for some compactly supported finite measure $\rho$. Therefore, the non-negativity of the second free cumulant is always preserved under the commutative point-wise multiplication, and therefore the property of being an ideal follows. 
\end{proof}
\begin{prop}
The family of free Poisson distributions $\mathbf{fP}$ with rate $\lambda\in\R_+$ and jump size $\alpha\in\R$ is closed with respect to the $\boxdot$-convolution, i.e. for $\lambda,\lambda'\in\R_+$ and $\alpha,\beta\in\R$ we have
$$
\nu_{\infty,\lambda,\alpha}\boxdot\nu_{\infty,\lambda',\beta}=\nu_{\infty,\lambda\lambda',\alpha\beta}.
$$
Therefore, $\mathbf{fP}$ is a $\boxdot$-monoid, and for $(\lambda,\alpha)\in\R^*_+\times\R^*$, an abelian group. 
\end{prop}
\begin{proof}
The statement follows from the representation given in equation~(\ref{inf_Poisson}) by component-wise multiplication.
\end{proof}

Let us remark that a more general statement should be valid which holds also for the {\bf compound free Poisson distribution}, cf.~[\cite{NS}~p.206]. 

We derived our analytic results for compactly supported probability measures, but it is natural to assume that one can extend them to all freely infinitely divisible probability measures.

\subsection{Endomorphisms}
As we have now established the Witt semi-ring structure for freely infinitely divisible probability measures, we shall investigate elements of its ring of endomorphisms in analogy with the theory of Witt vectors.

For $a\in\R$, we define the {\bf Teichmüller representative} $\tau$ as:
\begin{equation*}
\tau:\R\rightarrow \operatorname{ID}_c(\R,\boxplus),\quad a\mapsto\nu_{\infty,1,a},
\end{equation*}

which gives a group isomorphism 
$$
\tau:(\R^*,\cdot)\rightarrow (\nu_{\infty,1,\bullet},\boxdot).
$$

Let us define the squared {\bf shift} or {\bf décalage operator} $\hat{\partial}^2$.  For $\mu\in \operatorname{ID}_c(\R,\boxplus)$, with free cumulants $\kappa({\mu})=(\kappa_n({\mu}))_{n\in\N^*}$, let $\hat{\partial}^2(\mu)\in \operatorname{ID}_c(\R,\boxplus)$, such that for all $n\in\N^*$, we have:
$$
\kappa_n(\partial^2(\mu)):=\kappa_{n+2}(\mu),
$$
which by Lemma~\ref{shifted_cummulants} is well-defined. The notation $\hat{\partial}^2$ emphasises the fact that in order to have the décalage operator acting properly on measures we have to take its square, i.e. it is elliptic. 
\begin{prop}
The operator $\hat{\partial}^2$ defines a semi-ring endomorphism of $\operatorname{ID}_c(\R,\boxplus)$, i.e. for $\mu_1,\mu_2\in\operatorname{ID}_c(\R,\boxplus)$, we have:
\begin{eqnarray*}
\hat{\partial}^2(\mu_1\boxplus\mu_2)&=&\hat{\partial}^2(\mu_1)\boxplus \partial^2(\mu_2),\\
\hat{\partial}^2(\mu_1\boxdot\mu_2)&=&\hat{\partial}^2(\mu_1)\boxdot \hat{\partial}^2(\mu_2).
\end{eqnarray*}
In particular, it respects the additive and multiplicative unit, i.e. we have
\begin{eqnarray*}
\hat{\partial}^2(\delta_0)&=&\delta_0,\\
\hat{\partial}^2(\nu_{\infty,1,1})&=&\nu_{\infty,1,1}.
\end{eqnarray*}
\end{prop}
\begin{proof} 
From
$\kappa_n(\mu_1\boxplus\mu_2)=\kappa_n(\mu_1)+\kappa_n(\mu_2)$ and, by definition $\kappa_n(\mu_1\boxdot\mu_2)=\kappa_n(\mu_1)\cdot_{H}\kappa_n(\mu_2)$, $n\in\N^*$, the claim follows.
\end{proof}

We introduce the {\bf Frobenius endomorphism} as follows. For $a\in A^{\N^*}$, $n\in\N$, with $0^0:=1$, let $\mathbf{f}_n:A^{\N^*}\rightarrow A^{\N^*}$, and define  
\begin{equation}
\label{Frob_map}
(a_j)_{j\in\N^*}\mapsto (a_j^n)_{j\in\N^*}.
\end{equation}

\begin{prop}
For $\mu\in \operatorname{ID}_c(\R,\boxplus)$, $n\in\N$, we define
$$
\mathbf{f}_n(\mu):=R^{-1}(\mathbf{f}_n(\kappa(\mu))).
$$
Then the following holds:
\begin{equation}
\mathbf{f}_n(\mu)=\mu^{\boxdot n},
\end{equation}
and $\mathbf{f}_n$ yields a $\boxdot$-endomorphisms, i.e. for $\mu,\nu\in \operatorname{ID}_c(\R,\boxplus)$: 
\begin{eqnarray*}
\mathbf{f}_n(\mu\boxdot\nu)&=&\mathbf{f}_n(\mu)\boxdot\mathbf{f}_n(\nu),\\
(\mu\boxdot\nu)^{\boxdot n}&=&\mu^{\boxdot n}\boxdot\nu^{\boxdot n}.
\end{eqnarray*}
\end{prop}
\begin{proof}
$\mathbf{f}_n(\kappa(\mu))=(\kappa_j(\mu)^n)_{j\in\N^*}=\kappa(\mu)^{\boxdot n}$ from which the claim follows.
\end{proof}

We remark that the set of free Poisson distributions is invariant with respect to  $\partial^2$ and $\mathbf{f}_n$, in particular 
$$
\hat{\partial}^2(\nu_{\infty,\lambda,\alpha})=\nu_{\infty,\lambda\alpha^2,\alpha},
$$
holds.
\subsection{Geometric properties}
The space of infinitely divisible probability measures possesses, due to the Lévy-Khintchine  representation, a remarkable geometric structure. 

\begin{prop}
There exists a surjection $\pi:\operatorname{ID}_c(\R,\boxplus)\rightarrow \operatorname{Meas}_c(\R)$, with fibre 
$$
\pi^{-1}(\rho)\cong\R\times\{\rho\}\cong\R.
$$
The space $\operatorname{ID}_c(\R,\boxplus)$ carries two actions, namely for $\mu\in \operatorname{ID}_c(\R,\boxplus)$, $r\in\R$ and $c\in\R_+^*$, we have:
\begin{itemize}
\item the $\R$-action,(shift along the fibre) is given by:
$$
r.\mu:=\delta_{r}\boxplus\mu,
$$
\item
the $\R^*_+$-action (scaling) is given by: 
$$
c.\mu:=\nu_{\infty,c,1}\boxdot \mu.
$$
\end{itemize}
\end{prop}
\begin{proof}
$(r+r').\mu=\delta_{r+r'}\boxplus\mu=\delta_r\boxplus(\delta_{r'}\boxplus\mu)$ and $0.\mu=\delta_0\boxplus\mu=\mu$.
\end{proof}

So, for the Dirac, semi-circular and free Poisson distribution, with $c,r,\lambda\geq0$, and $\alpha\in\R$, we have
\begin{eqnarray*}
c.\delta_a&=& \delta_{c\cdot a},\\
c.\gamma_{a,r}&=& \gamma_{c\cdot a,{\sqrt{c}}\cdot r},\\
c.\nu_{\lambda,\alpha}&=& \nu_{c\cdot\lambda,\alpha},
\end{eqnarray*}
which shows that these classes of measures are invariant under the action. In particular, the semi-circular distributions form a convex sub-cone. 
\subsection{Relation with classical infinite divisibility}
Here we shall discuss the relation of our results with classically infinitely divisible measures, which due to the Bercovici-Pata bijection~\cite{BPB}, is feasible.
The main references for this section are~\cite{BNT,BPB,C2014}. Let, $\operatorname{ID}(\R,*)$ and $\operatorname{ID}(\R,\boxplus)$ denote the set of $*$-infinitely (classically) divisible and $\boxplus$-infinitely (freely) divisible probability measure on the real line, respectively. Let us emphasise, here we do not assume any restriction on the support.  

We have $\operatorname{ID}_c(\R,\boxplus)\subset\operatorname{ID}(\R,\boxplus)$ as a dense sub-monoid and that there is a bijection $\varphi$ between characteristic pairs $(\gamma,\sigma)$ and $\bullet$-infinitely divisible probability measures on $\R$, where $\bullet=*,\boxplus$, i.e.  we have
\begin{equation}
\label{LKpairs}
\varphi_{\operatorname{cp}}:\R\times\operatorname{Meas}(\R)\rightarrow\operatorname{ID}(\R,\bullet).
\end{equation}

Let $\mu\in\operatorname{ID}(\R,\boxplus)$ with characteristic pair $(\gamma,\sigma)$. Then from~[\cite{BNT} (4.3)] we obtain
\begin{equation}
\label{R-trafo_char_pair}
R_{\mu}(z)=\gamma+\int_{\R}\frac{z+x}{1-xz}d\sigma(x).
\end{equation}

\begin{prop}
Let $\pi_*:\operatorname{ID}(\R,*)\rightarrow\operatorname{Meas}(\R)$ and $\pi_{\boxplus}:\operatorname{ID}(\R,\boxplus)\rightarrow\operatorname{Meas}(\R)$ be the projections onto the second factor of the characteristic pair, i.e. $\pi_i(\mu):=\pr_2(\gamma,\sigma)=\sigma$, $i=*,\boxplus$ where $\mu$ is represented by $(\gamma,\sigma)$. The fibre $\pi^{-1}_i(\sigma)$ is isomorphic to $\R$, and the Bercovi-Pata bijection gives a bundle map $\varphi_{\operatorname{BP}}:\operatorname{ID}(\R,*)\rightarrow \operatorname{ID}(\R,\boxplus)$, i.e. the following diagram 
\[
\begin{xy}
  \xymatrix{
 \operatorname{ID}(\R,*)\ar[dr]_{\pi_*}\ar[rr]_{\varphi_{\operatorname{BP}}} & &\operatorname{ID}(\R,\boxplus)\ar[dl]^{\pi_{\boxplus}} \\
       & \operatorname{Meas}(\R)&
               }
\end{xy}
\]
commutes. 
\end{prop}

For $\mu,\nu\in\operatorname{ID}(\R,*)$ with classical cumulants $(c(\mu))_{n\in\N^*}$ and $(c(\nu))_{n\in\N^*}$, we define, equivalently to~(\ref{boxdot}), a binary operation $\star$ as follows:
\begin{equation}
\label{class_prod}
c(\mu\star\nu)_n:=c(\mu)_n\cdot c(\nu)_n,
\end{equation}
i.e. the two cumulant sequences are multiplied component-wise, and then the measure is obtained by the inverse Fourier transform of the exponential of the resulting sequence. 

We have:
\[
\begin{tabular}{l|r|l} distribution & parameters& classical cumulants   \\\hline 
Dirac & $a\in\R$ & (a,0,0,0,\dots) \\Normal & $m\in\R$, $\sigma^2\in\R^*_+$ & $(m,\sigma^2,0,0,\dots)$\\ 
Poisson & $\lambda\in\R_+^*$ & $(\lambda,\lambda,\lambda,\lambda,\dots)$ \end{tabular}
\]

\begin{prop}
The joint set of Dirac, normal and Poisson distributions form a commutative monoid for the operation $\star$, with unit the Poisson distribution with parameter $\lambda=1$.  
\end{prop}
\begin{rem}
Let us note that in equality~(\ref{generating_function}), the pair $(s_1,\rho)$ also uniquely characterises $\mu\in M_{\infty,\boxplus}(\R)$. From expression~(\ref{R-trafo_char_pair}) we obtain the relation $\kappa_1(\mu)=\gamma$, and for the measures $\sigma$ and $\rho$
$$
\int_{\R}\frac{z+x}{1-xz}d\sigma(x)=\int_{\R}\frac{z}{1-xz}d\rho(x).
$$
\end{rem}

\begin{prop}
The following statements hold:
\begin{enumerate}
\item
The sets $\operatorname{ID}(\R,\ast)$ and $\operatorname{ID}(\R,\boxplus)$ are convex cones for the following operation: for $\alpha,\beta\geq0$ and $\mu,\nu\in\operatorname{ID}(\R,\bullet)$ with $\bullet=*,\boxplus$,  let
$$
\alpha\mu+\beta\nu:=\varphi^{-1}(\alpha\gamma_{\mu}+\beta\gamma_{\nu},\alpha\mu+\beta\nu)
$$
where $\varphi_{\operatorname{cp}}^{-1}$ is the inverse of the bijection~(\ref{LKpairs}) .
\item The normal and the semicircular distributions form isomorphic convex sub-cones of $\operatorname{ID}(\R,\bullet)$.
\item
$\operatorname{ID}(\R,\bullet)$, $\bullet=*,\boxplus$, carries a fibre-wise $\R$-action, which is given by 
$$
\delta_a.\mu:=\varphi_{\operatorname{cp}}^{-1}((\gamma+a,\rho)),
$$
where $a\in\R$ and $\varphi_{\operatorname{cp}}^{-1}$ is the compositional inverse of the bijection(\ref{LKpairs}).
\end{enumerate}
\end{prop}

\section{Algebraic theories related to probability measures}
In the previous sections our investigation has been strongly motivated by the relation  free probability has with the ring of Witt vectors. Here we shall consider a more general framework, motivated by the (co)-monadic approach in~\cite{F2017}.

\subsection{Convex spaces}
Now we discuss the relations the barycentric calculus has with probability measures, in particular infinitely divisible ones. First, we state the following 

\begin{prop}
For $\mu,\nu\in\operatorname{Meas}(\R)$ such that all moments exists and $q\in[0,1]$, let $\lambda:=q\mu+(1-q)\nu\in \operatorname{Meas}(\R)$ be the {\bf $q$-convex linear combination} of the measures $\mu$ and $\nu$. Then the moment map $m$ is a morphism for convex linear combinations, i.e. the moments $(m_n(\lambda))_{n\in\N}$ satisfy:
$$
m_n(q\mu+(1-q)\nu)=qm_n(\mu)+(1-q) m_n(\nu).
$$ 
\end{prop}
\begin{proof}
This follows from 
$$
\int_{\R} x^n d(q\mu+(1-q)\nu)=q\int_{\R}x^nd\mu+(1-q)\int_{\R}x^nd\nu.
$$
 \end{proof}

Let us briefly recall the abstract notion of a {\bf convex space}, cf.~\cite{Fr,J}.
\begin{df}
\label{convex}
Let $C$ be as set. The structure of a {\bf convex space} on $C$ is given by, for every $p,q\in[0,1]$ a binary operation $+_q$, the {\bf barycentric addition}, which for $x,y,z\in C$ satisfies the following rules:
\begin{enumerate}
\item $+_q(x,y)=+_{1-q}(y,x),$
\item $+_q(x,x)=x,$
\item $+_0(x,y)=y,$
\item $+_p(+_q(x,y),z)=+_{p+(1-p)q}(+_{\frac{p}{p+(1-p)q}}(x,y),z)$,\quad if~$p+(1-p)q\neq0$.
\end{enumerate}
\end{df} 

\begin{prop}
For $\mu,\nu\in \operatorname{ID}_c(\R,\boxplus)$ and $\alpha,\beta\in\R_+$, define the family of binary operations $+_{\alpha,\beta}$ as follows:
$$
\mu+_{\alpha,\beta}\nu:=R^{-1}(\alpha R_{\mu}+\beta R_{\nu}),
$$
which implies
$$
\mu+_{\alpha,\beta}\nu=\mu^{\boxplus\alpha}\boxplus\nu^{\boxplus\beta}.
$$
Therefore, $(\operatorname{ID}_c(\R,\boxplus),+_{\alpha,\beta})$ has the structure of a convex cone, and for $\alpha+\beta=1$ of a convex space. 
\end{prop}
\begin{proof}
This follows from Lemma~\ref{convex_cones}. The second proof follows from the equality~(\ref{R-trafo_char_pair}).
\end{proof}
\begin{prop}
\label{meas_bary}
For $q\in[0,1]$, let $+_q:=+_{q,(1-q)}$. The operations $\boxplus$ and $\boxdot$ distribute over the barycentric addition $+_q$, i.e. for $\mu,\nu,\xi\in \operatorname{ID}_c(\R,\boxplus)$, we have:
\begin{eqnarray*}
(\mu+_q\nu)\boxplus\xi=(\mu\boxplus\xi)+_q(\nu\boxplus\xi),\\
(\mu+_q\nu)\boxdot\xi=(\mu\boxdot\xi)+_q(\nu\boxdot\xi).\\
\end{eqnarray*}
\end{prop}

\begin{proof}
In order to show distributivity, we calculate:
$$
R\left((\mu+_q\nu)\boxplus\xi\right)=(qR_{\mu}+(1-q)R_{\nu})+R_{\xi}=q(R_{\mu}+R_{\xi})+(1-q)(R_{\nu}+R_{\xi}),
$$
and
$$
R\left((\mu+_q\nu)\boxdot\xi\right)=(qR_{\mu}+(1-q)R_{\nu})\cdot_HR_{\xi}=q(R_{\mu}\cdot_HR_{\xi})+(1-q)(R_{\nu}\cdot_HR_{\xi}).
$$
\end{proof}

Let us point out that we have to distinguish between operations defined on moments and cumulants and in general we can not mix the two, as the following example shows, cf.~[\cite{NS} Remarks 12.10 (2)]. For $\mu:=\frac{1}{2}(\delta_{-1}+\delta_{+1})$, with the convex combination taken for moments, we have 
$$
\frac{1}{2}(\delta_{-1}+\delta_{+1})\boxplus\mu\neq\frac{1}{2}(\delta_{-1}\boxplus\mu)+\frac{1}{2}(\delta_{+1}\boxplus\mu).
$$

\subsection{The finite Giry monad}
Let $S$ be a semi-ring and $G_S$ be the functor
$$
G_S:\mathbf{Set}\rightarrow\mathbf{Set}
$$
such that for $X\in\mathbf{Set}$, we have 
$$
G_S(X):=\left\{\varphi:X\rightarrow S|~\text{$\varphi(x)=0$ for almost all $x\in X$, and}\sum_{x\in X}\varphi(x)=1\right\}.
$$
Then every $\varphi\in G_S(X)$ has a representation
$$
\varphi=s_1x_1+\cdots+s_nx_n,
$$
with $s_i\in S$, $x_i\in X$, $i=1,\dots,n$, and where $s_i=0$ is possible and the $x_i$ are not necessarily different, i.e. we allow also non-reduced formal sums, which can be considered as special $S$-valued divisors. 
\begin{prop}[\cite{J,Fr}]
Let $S$ be a commutative unital semiring. The {\bf distribution} or {\bf finite Giry monad} $(G_S,\eta,\mu)$ with coefficients in $S$, is given by: 
\begin{eqnarray*}
\eta_X&:&X\rightarrow G_S(X),\\
x&\mapsto& 1_Sx,\\
\mu_X&:&G_S^2:=G_S\circ G_S\rightarrow G_S,\\
\mu_X\left(\sum_{i=1}^n s_i\varphi_i\right)&:=&x\mapsto\sum_{i=1}^n s_i\varphi_i(x)
\end{eqnarray*}
\end{prop}
The corresponding statement for $\boxplus$-infinitely divisible measures is as follows.
\begin{prop}
$(\operatorname{ID}_c(\R,\boxplus),+_q, q\in[0,1])$ is an algebra for the finite Giry / distribution monad, with $\alpha:G_{\R_+}(\operatorname{ID}_c(\R,\boxplus))\rightarrow \operatorname{ID}_c(\R,\boxplus)$ given by:
$$ 
\alpha(\lambda_1\mu_1+\cdots\lambda_n x_n):=\mu_1+_{q_1}(\mu_2+_{\mu_2}(\cdots +_{q_{n-2}}(\mu_{n-1}+_{q_{n-1}}\mu_n)))
$$
for $\mu_i\in \operatorname{ID}_c(\R,\boxplus)$, $\lambda_i\in\R_+$, $\sum_{i=1}^n\lambda_i=1$, $q_1:=\lambda_1$ and $q_i:=\frac{q_i}{1-q_1-\dots-q_{i-1}}$, $i=1,\dots, n$.
\end{prop}
\begin{proof}
The statement follows from~[\cite{J} Theorem 4].
\end{proof}
\subsection{The convolution algebra }
We shall consider, following the notation in~\cite{MacL}, an algebraic theory $(\Omega,E)$ consisting of a graded set $\Omega$ of operations and $E$ a set of identities which together  define an $\langle\Omega,E\rangle$-algebra. Surprisingly, this algebra , which should abstractly capture the operations we have encountered so far, is also related to a number of other algebraic structures, such as differential algebras, endomorphism algebras and Witt vectors, cf.~\cite{F2017}.

\begin{df}
\label{pconv_alg}
The {\bf (partial) convolution algebra} is given by the following generators $\Omega$ and relations $E$:
\begin{itemize}
\item $\Omega(0):=\{0,1\}\times\R$, nullary operations,
\item $\Omega(1):=\{\id,\operatorname{inv},\partial, f_n\}$, $n\in\N$, unary operations, 
\item $\Omega(2):=\{+,\cdot,+_q\}$, $q\in\Delta^1$, where $\Delta^1$ denotes the geometric one-simplex, binary operations.
\end{itemize}
These satisfy the following relations $E$:
\begin{enumerate}
\item $(+,\cdot,\mathbf{0}:=(0,0),\mathbf{1}:=(1,1))$ forms a commutative semi-ring with additive unit $\mathbf{0}$ and multiplicative unit $\mathbf{1}$,
\item $+_q$ $q\in[0,1]$ satisfies the properties of Definition~\ref{convex},
\item ${\partial}$ is an endomorphism for $+,\cdot$ and $+_q$, $q\in[0,1]$,
\item $f_n$ $n\in\N^*$ is an endomorphism for $\cdot$.
\end{enumerate}
\end{df}

The above algebraic structure is not operadic, as its definition also involves co-operations. 
However, it contains several sub-operads, e.g. Definition~\ref{convex}, without property $2.$ defines an operad called the {\bf simplex operad}, as the derived $n$-ary operations also contain the standard $(n-1)$-simplex as label, cf. the essay by T.~Leinster~\cite{L2011}. This can be described as follows:

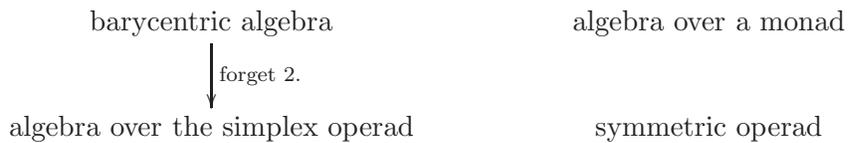
\[
\begin{xy}
  \xymatrix{
 \text{barycentric algebra}\ar[d]^{\text{forget $2.$}} & &\text{algebra over a monad} \\
\text{algebra over the simplex operad}       & &\text{symmetric operad}
               }
\end{xy}
\]

The set of operations we found fits into the algebraic structure which we summarise  next. Let us, define the operations $+:=\boxplus$, $\cdot:=\boxdot$, $\partial^2:=\hat{\partial}^2$, $\mathbf{f}_n$ as in~(\ref{Frob_map}) and $+_q$ as in Proposition~\ref{meas_bary}. Then we have
\begin{cor}
With the operations as defined above, the following holds:
\begin{enumerate}
\item
$(\operatorname{ID}_c(\R,\boxplus), \boxplus,\boxdot,\delta_0,\nu_{\infty,1,1}, \hat{\partial}^2,\mathbf{f}_n, n\in\N^*)$, has the structure of an algebra over the set operad generated by $\Omega(0),\Omega(1)$ and $+,\cdot$. 
\item
$(\operatorname{ID}_c(\R,\boxplus), \boxplus,\boxdot,+_q, q\in[0,1], \hat{\partial}^2,f_n,n\in\N)$ is an $(\Omega,E)$-algebra, with the operations and relations specified in Defintion~\ref{pconv_alg}.
\end{enumerate}
\end{cor}
\subsection*{Acknowledgements}
The author thanks the following people: John McKay for the discussions and his continuous interest, G. Cébron for the previous discussions and the (technically) helpful comments he made at several occasions and Roland Speicher for numerous inspiring discussions, his comments and questions and his support. Further he thanks the MPI in Bonn, for its hospitality.

\end{document}